\documentclass[12pt]{article}
\usepackage{amsmath}
\usepackage{amssymb}
\usepackage{amscd}
\usepackage{amsthm}

\theoremstyle{plain}
\newtheorem{Thm}{Theorem}

\newtheorem{Prop}[Thm]{Proposition}

\newtheorem{Lem}[Thm]{Lemma}

\theoremstyle{definition}

\newtheorem{Rem}[Thm]{Remark}

\numberwithin{equation}{section}

\title{Derived categories of toric varieties II}
\author{Yujiro Kawamata}

\begin{document}

\maketitle

\begin{abstract}
We prove two theorems on the derived categories of toric varieties, 
the existence of an exceptional collection consisting of sheaves 
for a divisorial extraction
and the finiteness of Fourier-Mukai partners.
\end{abstract}

\section{Introduction}

This paper supplements the first part of the series \cite{toric}, where 
we considered the semi-orthogonal 
decompositions of the derived categories with respect to the toric 
Minimal Model Program.
We proved that the semi-orthogonal complements for divisorial contractions and 
flips have full exceptional collections consisting of sheaves.

Akira Ishii and Kazushi Ueda pointed out that the divisorial extractions, 
another important class of 
birational transformations, were not yet treated.
Namely we did not consider toric birational morphisms $f: X \to Y$ 
between $\mathbf{Q}$-factorial projective toric varieties 
whose exceptional locus is a prime divisor $E$ such that $K_X+eE = f^*K_Y$ 
with $e > 0$.
We note that, if $e < 0$, then $f$ is a divisorial contraction 
which is already treated in \cite{toric}, and if $e = 0$, then $f$ is a
log crepant morphism and we have a derived equivalence (\cite{log crepant}). 
So we consider this case in \S 1, and prove that the semi-orthogonal 
complement has again a full exceptional collection consisting of sheaves.
It is rather remarkable because the directions are opposite for 
the fully faithful embedding functors between derived 
categories in the cases $e > 0$ and $e < 0$.

We also correct certain notation in \cite{toric}~\S 5 (paragraph before 
Remark 5.1)
according to a remark of Akira Ishii and Kazushi Ueda.
Namely we write $j_{1*}j_2^*$ instead of $j^*$ because there is no morphism 
of stacks over a morphism of schemes $D \to X$, 
but we have a flip-like diagram 
$\mathcal{D} \leftarrow \tilde{\mathcal{D}} \rightarrow \mathcal{X}$
(paragraphs between Lemmas 2 and 3).

In \S 2, we answer a question by Shinnosuke Okawa raised at 
Chulalongkorn University conference.
We prove that the number of Fourier-Mukai partners of a 
$\mathbf{Q}$-factorial projective toric variety is finite, confirming a 
conjecture which is related to the finiteness conjecture of the minimal models 
(\cite{DK}).

The author would like to thank Akira Ishii, Kazushi Ueda and Shinnosuke Okawa.
The author would also like to thank the referee for careful reading.

\section{Divisorial extraction}

We proved in \cite{log crepant} that the Minimal Model Program in the 
category of toroidal varieties yields semi-orthogonal decompositions 
of derived categories.
In \cite{toric} we proved in the toric case 
that the semi-orthogonal complements have
full exceptional collections consisting of sheaves.
We extend this result to another case of important toric morphism, a toric
divisorial extraction.

\begin{Thm}
Let $\phi: X \to Y$ be a birational morphism between $\mathbf{Q}$-factorial 
projective 
toric varieties such that the exceptional locus is a prime divisor denoted by 
$D$.
Let $B$ be an effective torus-invariant $\mathbf{Q}$-divisor on $X$ 
whose coefficients belong to a set 
\[
\{\frac{r-1}r \mid r \in \mathbf{Z}_{>0}\}
\]
and $C = \phi_*B$.
Let $\pi_X: \mathcal{X} \to X$ and $\pi_Y: \mathcal{Y} \to Y$ 
be the natural morphisms from smooth Deligne-Mumford stacks corresponding 
to the pairs $(X,B)$ and $(Y,C)$ respectively.
Assume that 
\[
K_X+B<\phi^*(K_Y+C).
\] 
Then there is a fully faithful triangulated functor 
\[
\Phi: D^b(\text{Coh}(\mathcal{X})) \to D^b(\text{Coh}(\mathcal{Y}))
\]
such that the semi-orthogonal complement of the image 
$\Phi(D^b(\text{Coh}(\mathcal{X})))^{\perp}$ has a full exceptional collection
consisting of sheaves.
\end{Thm}

\begin{proof}
This is the case which was not considered in the previous paper 
\cite{toric}, where we considered only the case $K_X+B \ge \phi^*(K_Y+C)$.
The proof is very similar, but we need additional calculations.

We use the notation of \cite{toric} whenever it is possible.
Especially the morphism $\phi: X \to Y$ is controlled by an equation 
\[
a_1v_1 + \cdots + a_{n+1}v_{n+1}=0
\]
locally over $Y$, where the $v_i$ are primitive vertices of cones in a
decomposition of a cone corresponding 
to a toric affine open subset $U$ of $Y$, 
and the coefficients of the equation $a_i$ are coprime integers.
Since $\phi$ is a divisorial contraction, 
we have $a_i \ge 0$ except for $i=n+1$.
We assume that $a_i > 0$ for $1 \le i \le \alpha$, $a_i=0$ 
for $\alpha < i \le n$, and $a_{n+1}<0$.

Let $D_i$ be the prime divisors on $X$ corresponding to the vertices $v_i$,
and $E_i = \phi(D_i)$ for $1 \le i \le n$ their images on $Y$.
$D=D_{n+1}$ is the exceptional divisor of $\phi$, and the restricted morphism 
$\bar{\phi}: D \to F = \phi(D)$ is a toric Mori fiber space.
Let $\mathcal{D}_i$ and $\mathcal{E}_i$ respectively be the prime divisors on 
$\mathcal{X}$ and $\mathcal{Y}$ above $D_i$ and $E_i$.

We write $B\vert_{\phi^{-1}(U)} =\sum_{i=1}^{n+1}\frac{r_i-1}{r_i}D_i$, 
and $C\vert_U = \sum_{i=1}^n\frac{r_i-1}{r_i}E_i$.
The assumption that $K_X+B$ is positive for $\phi$ is expressed by 
the following inequality
\[
\sum_{i=1}^{n+1} \frac{a_i}{r_i} < 0.
\]

The following Lemma~\ref{2} is \cite{log crepant} Theorem 4.2~(4):

\begin{Lem}\label{2}
Let $\mathcal{W}=(\mathcal{X} \times_Y \mathcal{Y})\tilde{}$ be 
the normalized fiber product, and let $\mu$ and $\nu$ be the projections.
Then the functor 
$\Phi = \nu_*\mu^*: D^b(\text{Coh}(\mathcal{X})) 
\to D^b(\text{Coh}(\mathcal{Y}))$ is fully faithful.
Moreover, the invertible sheaves $\mathcal{O}_{\mathcal{X}}
(\sum_{i=1}^{n+1} k_i \mathcal{D}_i)$ for the sequences of integers 
$k=(k_1,\dots,k_{n+1})$ such that 
\[
0 \le - \sum_{i=1}^{n+1} \frac{a_ik_i}{r_i} 
< \sum_{i=1}^{\alpha} \frac{a_i}{r_i}
\]
span the triangulated category $D^b(\text{Coh}(\mathcal{X}))$, and 
\[
\Phi(\mathcal{O}_{\mathcal{X}}(\sum_{i=1}^{n+1} k_i \mathcal{D}_i)) \cong
\mathcal{O}_{\mathcal{Y}}(\sum_{i=1}^n k_i \mathcal{E}_i)
\]
for such sequences of integers.
\end{Lem}

\begin{proof}
We recall the proof for the convenience of the reader.
We have
\[
\sum_{i=1}^{n+1} k_i \mu^*\mathcal{D}_i
\equiv \sum_{i=1}^n k_i \nu^*\mathcal{E}_i
+ \frac{r_{n+1}}{a_{n+1}} \sum_{i=1}^{n+1} \frac{a_ik_i}{r_i} 
\mu^*\mathcal{D}_{n+1}
\]
and 
\[
K_X+\sum_{i=1}^n \frac{r_i-1}{r_i}D_i + D_{n+1}
= \phi^*(K_Y+\sum_{i=1}^n \frac{r_i-1}{r_i}E_i) 
- \sum_{i=1}^n \frac{a_i}{r_i}\frac{1}{a_{n+1}}D_{n+1}.
\]
Since 
\[
\frac{r_{n+1}}{a_{n+1}} \sum_{i=1}^{n+1} \frac{a_ik_i}{r_i}
< - \frac{r_{n+1}}{a_{n+1}} \sum_{i=1}^n \frac{a_i}{r_i}
\]
we have 
\[
R^q\nu_*\mu^*\mathcal{O}_{\mathcal{X}}(\sum_{i=1}^{n+1} k_i \mathcal{D}_i)
\cong 0
\]
for $q > 0$. 
Moreover since 
\[
\frac{r_{n+1}}{a_{n+1}} \sum_{i=1}^{n+1} \frac{a_ik_i}{r_i} \ge 0
\]
we have 
\[
\nu_*\mu^*\mathcal{O}_{\mathcal{X}}(\sum_{i=1}^{n+1} k_i \mathcal{D}_i)
\cong \mathcal{O}_{\mathcal{Y}}(\sum_{i=1}^n k_i \mathcal{E}_i).
\]

For another such sequence $k'=(k'_1,\dots,k'_{n+1})$, we have
\[
\frac{a_{n+1}}{r_{n+1}} < - \sum_{i=1}^n \frac{a_i}{r_i}
< - \sum_{i=1}^{n+1} \frac{a_i(k_i-k'_i)}{r_i}
< \sum_{i=1}^n \frac{a_i}{r_i}
\]
hence
\[
R^q\nu_*\mu^*\mathcal{O}_{\mathcal{X}}(\sum_{i=1}^{n+1} (k_i-k'_i) 
\mathcal{D}_i) \cong 0
\]
for $q > 0$ and 
\[
\nu_*\mu^*\mathcal{O}_{\mathcal{X}}(\sum_{i=1}^{n+1} (k_i-k'_i) 
\mathcal{D}_i)
\cong \mathcal{O}_{\mathcal{Y}}(\sum_{i=1}^n (k_i-k'_i) 
\mathcal{E}_i).
\]
Thus the natural homomorphism
\[
\text{Hom}(L',L) \to \text{Hom}(\Phi(L'), \Phi(L'))
\]
is bijective for $L=\mathcal{O}_{\mathcal{X}}(\sum_{i=1}^{n+1} k_i 
\mathcal{D}_i)$ and $L'=\mathcal{O}_{\mathcal{X}}(\sum_{i=1}^{n+1} k'_i 
\mathcal{D}_i)$. 

Since $\bigcup_{i=1}^{\alpha} \mathcal{D}_i = \emptyset$, 
the following Koszul complex is exact:
\[
0 \to \mathcal{O}_{\mathcal{X}}(-\sum_{i=1}^{\alpha} \mathcal{D}_i) 
\to \dots \to \sum_{i=1}^{\alpha} \mathcal{O}_{\mathcal{X}}(-\mathcal{D}_i)
\to \mathcal{O}_{\mathcal{X}} \to 0.
\]
Therefore the sheaves $\mathcal{O}_{\mathcal{X}}(\sum_{i=1}^{n+1} k_i 
\mathcal{D}_i)$ whose coefficients $k_i$ 
satisfy the assumption of the lemma generate the same 
category as those whose coefficients are general.
\end{proof}

We shall construct a full exceptional collection on the semiorthogonal 
complement of the image of $\Phi$.

We define a $\mathbf{Q}$-divisor 
$\bar B = \sum_{i=1}^n \frac{r_it_i-1}{r_it_i}\bar D_i$ for 
$\bar D_i=D_i \cap D$ similarly as in \cite{toric}~\S 5.
Namely, we write 
\[
v_i \equiv t_i\bar v_i
\]
in $N_D=N_X/\mathbf{Z}v_{n+1}$ 
for primitive vectors $\bar v_i \in N_D$, and $a_it_i=t\bar a_i$ such that the 
$\bar a_i$ for $1 \le i \le \alpha$ are coprime integers, where $N_X$ and 
$N_D$ are lattices corresponding to the toric varieties $X$ and $D$ 
respectively.
Here we note that 
\[
M_D = \{m \in M_X \mid \langle m, v_{n+1} \rangle = 0\}
\]
for the dual lattices $M_X$ and $M_D$ of $N_X$ and $N_D$ respectively.

The toric morphism $\bar{\phi}: D \to F$ is controlled locally over $F$ 
by an equation 
\[
\sum_{i=1}^n \bar a_i\bar v_i=0
\]
where we have $\bar a_i=0$ for $\alpha < i \le n$. 
We have $D_i \vert_D = \frac 1{t_i} \bar D_i$ as $\mathbf{Q}$-Cartier divisors.

Let $\pi_D: \mathcal{D} \to D$ be the natural morphism from a smooth 
Deligne-Mumford stack corresponding to the pair $(D,\bar B)$.
Let $\tilde{\mathcal{D}} = (\mathcal{X} \times_X D)\tilde{}$ be the normalized
fiber product.
Then there are natural morphisms $j_1: \tilde{\mathcal{D}} \to \mathcal{D}$
and $j_2: \tilde{\mathcal{D}} \to \mathcal{X}$ satisfying
\[
j_{1*}j_2^*\mathcal{O}_{\mathcal{X}}(\mathcal{D}_i) 
\cong \mathcal{O}_{\mathcal{D}}(\bar{\mathcal{D}}_i)
\]
for $i = 1,\dots,n$, 
where $\mathcal{D}_i$ and $\bar{\mathcal{D}}_i$ are prime divisors on 
$\mathcal{X}$ and $\mathcal{D}$ corresponding to $D_i$ and $\bar D_i$
respectively. 

Next we define a $\mathbf{Q}$-divisor $\bar C$ on $F$ by 
$\bar C= \sum_{i=\alpha+1}^n \frac{r_is_it_i-1}{r_is_it_i}\bar E_i$ for 
$\bar E_i=E_i \cap F$ similarly as in \cite{toric}~\S4.
Namely, we write 
\[
\bar v_i \equiv s_i\tilde v_i
\]
for primitive vectors 
$\tilde v_i$ in the lattice 
\[
N_F = N_D/(\sum_{i=1}^{\alpha} \mathbf{R}v_i \cap N_D)
\]
corresponding to $F$.
Here we note that 
\[
M_F = \{m \in M_F \mid \langle m, v_i \rangle = 0, i=1,\dots,\alpha \}
\]
for the dual lattices $M_F$ of $N_F$.

Let $\pi_F: \mathcal{F} \to F$ be the natural morphism from a smooth 
Deligne-Mumford stack corresponding to the pair $(F,\bar C)$.
Let $\tilde{\mathcal{F}} = (\mathcal{Y} \times_Y F)\tilde{}$ be the normalized
fiber product.
Then there are natural morphisms $j_{F,1}: \tilde{\mathcal{F}} \to \mathcal{F}$
and $j_{F,2}: \tilde{\mathcal{F}} \to \mathcal{Y}$ satisfying
\[
j_{F,1*}j_{F,2}^*\mathcal{O}_{\mathcal{Y}}(\mathcal{E}_i) 
\cong \mathcal{O}_{\mathcal{F}}(\bar{\mathcal{E}}_i)
\]
for $i = \alpha+1,\dots,n$, 
where $\mathcal{E}_i$ and $\bar{\mathcal{E}}_i$ are prime divisors on 
$\mathcal{Y}$ and $\mathcal{F}$ corresponding to $E_i$ and $\bar E_i$
respectively.
Indeed we have an equality $E_i \vert_F = \frac 1{s_it_i} \bar E_i$
for $\alpha < i \le n$, 
which is confirmed by the following equalities $\phi^*E_i=D_i$, 
$D_i \vert_D=\frac 1{t_i}\bar D_i$ and $\bar{\phi}^*\bar E_i = s_i\bar D_i$.

We have an induced morphism $\bar{\psi}: \mathcal{D} \to \mathcal{F}$. 
Since $\bar{\psi}$ is smooth by \cite{toric}~Corollary~4.2, we have
\[
\bar{\psi}^*\mathcal{O}_{\mathcal{F}}(\bar{\mathcal{E}}_i) 
\cong \mathcal{O}_{\mathcal{D}}(\bar{\mathcal{D}}_i)
\]
for $\alpha < i \le n$.

\begin{Lem}\label{3}
Let $k_1, \dots, k_n$ be integers, and define $k_{n+1}$ by
an equation
\[
\sum_{i=1}^{n+1} \frac{a_ik_i}{r_i} = 0.
\]
If $k_{n+1}$ is not an integer which is divisible by $r_{n+1}$, then 
\[
j_{F,1*}j_{F,2}^*\mathcal{O}_{\mathcal{Y}}
(\sum_{i=1}^{\alpha}k_i\mathcal{E}_i) \cong 0.
\]
\end{Lem}

\begin{proof}
Since $j_{F,2}^*\mathcal{O}_{\mathcal{Y}}
(\sum_{i=1}^{\alpha}k_i\mathcal{E}_i)$ is an invertible sheaf, 
its direct image sheaf is either an invertible sheaf or a zero sheaf on 
$\mathcal{F}$.
If it is an invertible sheaf, then its pull-back 
\[
\bar{\psi}^*j_{F,1*}j_{F,2}^*\mathcal{O}_{\mathcal{Y}}
(\sum_{i=1}^{\alpha}k_i\mathcal{E}_i)
\]
is also an invertible sheaf, which should be of the form
\[
j_{1*}j_2^*\mathcal{O}_{\mathcal{X}}(\sum_{i=1}^{\alpha}k_i\mathcal{D}_i
+k_{n+1}\mathcal{D}_{n+1})
\]
for an integer $k_{n+1}$ such that
$\sum_{i=1}^{n+1} \frac{a_ik_i}{r_i} = 0$.
Since it is non-zero if and only if $k_{n+1}$ is divisible by 
$r_{n+1}$, we conclude our proof.
\end{proof}

Our theorem is a consequence of the following Proposition~\ref{5} 
combined with \cite{toric}~Theorem~1.1. 

\begin{Prop}\label{5}
(1) The functor 
\[
j_{F,2*}j_{F,1}^*: D^b(\text{Coh}(\mathcal{F})) 
\to D^b(\text{Coh}(\mathcal{Y}))
\]
is fully faithful.

\vskip .5pc 

Let $D^b(\text{Coh}(\mathcal{F}))_k$ denote the full subcategory of 
$D^b(\text{Coh}(\mathcal{Y}))$ defined by
\[
D^b(\text{Coh}(\mathcal{F}))_k = j_{F,2*}j_{F,1}^*D^b(\text{Coh}(\mathcal{F})) 
\otimes \mathcal{O}_{\mathcal{Y}}(\sum_{i=1}^{\alpha} k_i\mathcal{E}_i)
\]
for a sequence of integers $k=(k_1,\dots,k_{\alpha})$.
We set $k_{\alpha+1}=\dots=k_n=0$ when necessary.

\vskip .5pc 

(2) If 
\[
0 < \sum_{i=1}^{n+1} \frac{a_ik_i}{r_i} 
\le -\sum_{i=1}^{n+1} \frac{a_i}{r_i}
\]
for some integer $k_{n+1}$, then 
\[
\text{Hom}^q(\Phi(D^b(\text{Coh}(\mathcal{X}))), 
D^b(\text{Coh}(\mathcal{F}))_k) = 0
\]
for all $q$.

\vskip .5pc 

(3) Let $k'=(k'_1,\dots,k'_{\alpha})$ be another sequence of integers such that
\[
0 < \sum_{i=1}^{n+1} \frac{a_i(k'_i-k_i)}{r_i} 
< - \sum_{i=1}^{n+1} \frac{a_i}{r_i}
\]
for some integers $k_{n+1}$ and $k'_{n+1}$.
Then 
\[
\text{Hom}^q(D^b(\text{Coh}(\mathcal{F}))_k, 
D^b(\text{Coh}(\mathcal{F}))_{k'}) = 0
\]
for all $q$.

\vskip .5pc 

(4) The subcategories $\Phi(D^b(\text{Coh}(\mathcal{X})))$ and 
the $D^b(\text{Coh}(\mathcal{F}))_k$ for 
\[
0 < \sum_{i=1}^{n+1} \frac{a_ik_i}{r_i} 
\le -\sum_{i=1}^{n+1} \frac{a_i}{r_i}
\]
for some integers $k_{n+1}$ 
generate $D^b(\text{Coh}(\mathcal{Y}))$ as a triangulated category.
\end{Prop}

\begin{proof}
(1) We shall prove that the natural homomorphisms
\[
\text{Hom}^q(L,L') \to \text{Hom}^q(j_{F,2*}j_{F,1}^*L,j_{F,2*}j_{F,1}^*L')
\cong \text{Hom}^q(j_{F,1*}j_{F,2}^*j_{F,2*}j_{F,1}^*L,L')
\]
are bijective for all $q$ and for all invertible sheaves 
$L$ and $L'$ on $\mathcal{F}$.

Let $V = \sum_{i=1}^{\alpha} \mathcal{O}_{\mathcal{Y}}(-\mathcal{E}_i)$.
Then there is a Koszul resolution of 
$j_{F,2*}j_{F,1}^*\mathcal{O}_{\mathcal{F}}$:
\[
0 \to \bigwedge^{\alpha}V \to \dots \to V \to 
\mathcal{O}_{\mathcal{Y}} \to j_{F,2*}j_{F,1}^*\mathcal{O}_{\mathcal{F}} \to 0.
\]
Thus 
\[
H_q(j_{F,1*}j_{F,2}^*j_{F,2*}j_{F,1}^*\mathcal{O}_{\mathcal{F}}) 
\cong \bigoplus_{\# I = q} j_{F,1*}j_{F,2}^*
\mathcal{O}_{\mathcal{Y}}(-\sum_{i \in I}\mathcal{E}_i)
\]
where the $I$ run all the subsets of $\{1,\dots,\alpha\}$ such that 
$\# I = q$.
Since $\sum_{i=1}^{n+1} \frac{a_i}{r_i} < 0$, we have 
\[
0 < \sum_{i \in I} \frac{a_i}{r_i} < \frac{\vert a_{n+1}\vert}{r_{n+1}}
\] 
for $q \ne 0$.
Therefore we have 
\[
H_q(j_{F,1*}j_{F,2}^*j_{F,2*}j_{F,1}^*\mathcal{O}_{\mathcal{F}}) \cong 0
\]
for such $q$ by Lemma~\ref{3}.
By the projection formula, we conclude that
\[
j_{F,1*}j_{F,2}^*j_{F,2*}j_{F,1}^*L \cong L
\]
hence the assertion.

\vskip 1pc

(2) By Lemma~\ref{2} or \cite{log crepant} Theorem 4.2~(4), 
$\Phi(D^b(\text{Coh}(\mathcal{X}))$
is spanned by invertible sheaves
$\mathcal{O}_{\mathcal{Y}}(\sum_{i=1}^n l_i\mathcal{E}_i)$
for
\[
0 \le - \sum_{i=1}^{n+1} \frac{a_il_i}{r_i} 
< \sum_{i=1}^{\alpha} \frac{a_i}{r_i}
\]
where $l_{n+1}$ are some integers.
By adding the inequalities, we obtain
\[
0 < - \sum_{i=1}^{n+1} \frac{a_i(l_i-k_i)}{r_i} < - \frac{a_{n+1}}{r_{n+1}}.
\]
Therefore for any invertible sheaf $L$ on $\mathcal{F}$, we have
\[
\begin{split}
&\text{Hom}^q
(\mathcal{O}_{\mathcal{Y}}(\sum_{i=1}^n l_i\mathcal{E}_i),
j_{F,2*}j_{F,1}^*L \otimes 
\mathcal{O}_{\mathcal{Y}}(\sum_{i=1}^{\alpha} k_i\mathcal{E}_i)) \\
&\cong \text{Hom}^q
(j_{F,1*}j_{F,2}^*
\mathcal{O}_{\mathcal{Y}}(\sum_{i=1}^n (l_i-k_i)\mathcal{E}_i),L) 
\cong 0
\end{split}
\]
for all $q$ by Lemma~\ref{3}.

\vskip 1pc

(3) For any fixed subset $I \subset \{1,\dots,\alpha\}$, we set
$\epsilon_i = 1$ or $0$ according to whethre $i \in I$ or not.
Then we have
\[
0 > \sum_{i=1}^{n+1} \frac{a_i(k_i-k'_i-\epsilon_i)}{r_i} 
> \frac{a_{n+1}}{r_{n+1}}.
\]
Therefore for any invertible sheaves $L$ and $L'$ on $\mathcal{F}$, we have
\[
\begin{split}
&\text{Hom}^q
(j_{F,2*}j_{F,1}^*L \otimes 
\mathcal{O}_{\mathcal{Y}}(\sum_{i=1}^{\alpha} k_i\mathcal{E}_i),
j_{F,2*}j_{F,1}^*L' \otimes 
\mathcal{O}_{\mathcal{Y}}(\sum_{i=1}^{\alpha} k'_i\mathcal{E}_i)) 
\\
&\text{Hom}^q
(j_{F,1*}j_{F,2}^*j_{F,2*}j_{F,1}^*L \otimes \mathcal{O}_{\mathcal{Y}}
(\sum_{i=1}^{\alpha} (k_i-k'_i)\mathcal{E}_i),L') 
\cong 0
\end{split}
\]
for all $q$ by Lemma~\ref{3}.

\vskip 1pc

(4) Let $T$ be the triangulated subcategory of $D^b(\text{Coh}(\mathcal{Y}))$
generated by $\Phi(D^b(\text{Coh}(\mathcal{X})))$ and the 
$D^b(\text{Coh}(\mathcal{F}))_k$ for 
\[
0 < \sum_{i=1}^{n+1} \frac{a_ik_i}{r_i} 
\le -\sum_{i=1}^{n+1} \frac{a_i}{r_i}
\]
for some integers $k_{n+1}$.
We shall prove that $T$ contains all invertible sheaves of the form
\[
\mathcal{O}_{\mathcal{Y}}(\sum_{i=1}^n k_i \mathcal{E}_i).
\]

By Lemma~\ref{2}, $T$ contains invertible sheaves
$\mathcal{O}_{\mathcal{Y}}(\sum_{i=1}^n k_i \mathcal{E}_i)$ if
\[
- \sum_{i=1}^{\alpha} \frac{a_i}{r_i} < \sum_{i=1}^{n+1} \frac{a_ik_i}{r_i} 
\le 0
\]
for some integers $k_{n+1}$.
On the other hand, for an arbitrary sequence of integers $(k_1,\dots,k_n)$, 
there exists an integer $k_{n+1}$ such that 
\[
- \sum_{i=1}^{\alpha} \frac{a_i}{r_i} < 
\sum_{i=1}^{n+1} \frac{a_ik_i}{r_i} \le 
- \sum_{i=1}^{n+1} \frac{a_i}{r_i}.
\]
The difference of the above two intervals 
is covered by the consideration on the Koszul resolution of 
the sheaf $j_{F,2*}j_{F,1}^*\mathcal{O}_{\mathcal{F}}$.
Indeed assume that $k=(k_1,\dots,k_{n+1})$ is a sequence of integers such that
\[
0 < \sum_{i=1}^{n+1} \frac{a_ik_i}{r_i} \le 
- \sum_{i=1}^{n+1} \frac{a_i}{r_i}. 
\]
If $T$ contains all the invertible sheaves of the form 
\[
\mathcal{O}_{\mathcal{Y}}(\sum_{i=1}^n l_i \mathcal{E}_i)
\]
such that $l_i = k_i$ or $k_i - 1$, $l_{n+1}=k_{n+1}$ and 
that $\sum_{i=1}^n l_i < \sum_{i=1}^n k_i$, then $T$ also contains 
$\mathcal{O}_{\mathcal{Y}}(\sum_{i=1}^n k_i \mathcal{E}_i)$, 
because $D^b(\text{Coh}(\mathcal{F}))_k$ is contained in $T$.
Therefore by the induction on $\sum_{i=1}^{n+1} k_i$, we obtain the
assertion.
\end{proof}

If $k = (k_1,\dots,k_{n+1})$ and $k' = (k'_1, \dots, k'_{n+1})$ are
sequences of integers such that 
\[
\sum_{i=1}^{n+1} \frac{a_ik_i}{r_i} = 
\sum_{i=1}^{n+1} \frac{a_ik'_i}{r_i}
\]
then $j_{F,1*}j_{F,2}^*\mathcal{O}_{\mathcal{Y}}
(\sum_{i=1}^{\alpha} (k_i-k'_i)\mathcal{E}_i)$ is either an invertible sheaf 
on $\mathcal{F}$ or $0$.
In the former case we have 
$D^b(\text{Coh}(\mathcal{F}))_k = D^b(\text{Coh}(\mathcal{F}))_{k'}$, while 
in the latter case we have 
\[
\text{Hom}^q(D^b(\text{Coh}(\mathcal{F}))_k, 
D^b(\text{Coh}(\mathcal{F}))_{k'}) = 0
\]
for all $q$.
Therefore we obtain an exceptional collection of the semiorthogonal 
complement.
\end{proof}

%%%%%%%%%%%%%%%%%%%%%%%%%%%%%%%%%%%%%%%%%%%%%%

\section{Fourier-Mukai partners}

In the minimal model program, we conjectured that, given a smooth projective
variety, or more generally a smooth projective pair, there exist only 
finitely many minimal models, or log minimal models, up to isomorphisms which
are birationally equivalent to the given variety.
In the derived setting, we conjecture that, given a smooth projective
variety, there exist only finitely many smooth projective
varieties up to isomorphisms whose derived categories are equivalent to
the given one (\cite{DK}).
We also expect more generalized statement to hold for projective varieties with
only quotient singularities.
We add here one more example which confirms this conjecture.

\begin{Thm}
Let $X$ be a projective $\mathbf{Q}$-factorial toric variety and 
$Y$ a projective variety which has only quotient singularities.
Let $\mathcal{X}$ and $\mathcal{Y}$ be smooth Deligne-Mumford stacks 
associated to $X$ and $Y$ respectively.
Assume that there exists an equivalence of triangulated categories 
$\Phi: D^b(\mathcal{X}) \cong D^b(\mathcal{Y})$.
Then $Y$ is also a projective $\mathbf{Q}$-factorial toric variety, and 
the kernel object for $\Phi$ induces a toric birational map 
$\phi: X \dashrightarrow Y$. 
In particular, $Y$ has only abelian quotient singularities.
Moreover there exist only finitely many such birational maps when $X$ is fixed
and $Y$ is varied.
\end{Thm}

\begin{proof} 
Let $E \in D^b(\mathcal{X} \times \mathcal{Y})$ be the kernel of $\Phi$, 
an object giving the equivalence $\Phi$ (\cite{equiv}).
Then we have an isomorphism 
\[
E \otimes p_1^*\omega_{\mathcal{X}} \cong E \otimes p_2^*\omega_{\mathcal{Y}}
\]
where $p_1$ and $p_2$ are projections.

Let $T$ be the torus contained in $X$ and $B = X \setminus T$.
We take a general point $y \in \mathcal{Y}$ such that the support of 
$E_y=\Psi(\mathcal{O}_y)$ contains a point in $T$, where 
$\Psi: D^b(\mathcal{Y}) \cong D^b(\mathcal{X})$ is the equivalence given by
$\Psi(a) = p_{1*}(p_2^*a \otimes E)$.
Since $K_X+B \sim 0$, we have an isomorphism 
\[
E_y \cong E_y \otimes \mathcal{O}_{\mathcal{X}}(-B).
\]  
It follows that the support of $E_y$ is a point, say $x \in T$, by the 
same argument as in the proof of \cite{DK}~Theorem~2.3~(2).

If we take another such point $y'$, then $E_{y'}$ is supported 
by a different point $x' \in T$, 
because we have $\text{Hom}(\mathcal{O}_y,\mathcal{O}_{y'}[p])=0$ for all $p$.
Therefore the support of $E$ gives a birational map 
$\phi: X \dashrightarrow Y$.
Let $Z \subset X \times Y$ be the graph of $\phi$.
By the first isomorphism, we have
$q_1^*K_X = q_2^*K_Y$, where $q_1: Z \to X$ and $q_2: Z \to Y$ 
are projections, 
i.e., $X$ and $Y$ are $K$-equivalent.

Let us consider the case where $\phi$ is not surjective in codimension $1$.
Let $D_Y$ be a prime divisor on $Y$ whose center on $X$ is not a divisor.
Since $X$ and $Y$ are $K$-equivalent, there exists a crepant 
divisorial extraction 
$\alpha: X' \to X$ whose exceptional divisor is the prime divisor $D_X$ 
corresponding to $D_Y$.
Indeed there exists a toric projective 
birational morphism $\beta: X'' \to X$ such that
all divisorial valuations whose log discrepancies are at most $1$ appear as
prime divisors on $X''$.
We note that $\beta$ is toric because it is obtained by blowing up 
a relative minimal model of $X$, which is toric, at centers which are 
also toric.
Then we can construct $\alpha$ by contracting all exceptional divisors except 
$D_X$.

Since $\alpha$ is a crepant toric morphism, we have an equivalence
$D^b(\mathcal{X}) \cong D^b(\mathcal{X'})$ for the smooth Deligne-Mumford 
stack $\mathcal{X'}$ associated to $X'$ (\cite{log crepant}).
By replacing $X$ by $X'$ if necessary, 
we may assume that $\phi$ is surjective in 
codimension $1$.

By the paragraph after \cite{HK}~ Corollary~2.4, $X$ is a Mori dream space.
By Definition~1.10 of loc. cit., all birational maps $\phi$ from $X$ 
which are surjective in codimension $1$ are obtained in the following way:
there exist finitely many birational
maps $f_i: X \dashrightarrow X_i$ corresponding to the 
chambers of the movable cone $\text{Mov}(X)$, 
and $\phi$ coincides with the composition of $f_i$ and a morphism from $X_i$ 
corresponding to one of the finitely many faces of the nef cone 
$\text{Nef}(X_i)$.
Moreover all of them are toric maps and morphisms to toric varieties.
Therefore we proved the finiteness of the birational maps $\phi$.
\end{proof}

\begin{Rem}
More generally, the same statement with the same proof holds if 
$K_X$ or $-K_X$ supports a big divisor and 
$X$ as well as its crepant blowing-ups are Mori dream spaces .
\end{Rem}

Graduate School of Mathematical Sciences, University of Tokyo,
Komaba, Meguro, Tokyo, 153-8914, Japan 

Department of Mathematics, Faculty of Science, King Abdulaziz University,  
P. O. Box 80257, Jeddah 21589, Saudi Arabia

kawamata@ms.u-tokyo.ac.jp

\end{document}